\documentclass[a4paper,11pt]{amsart}
\pdfoutput=1 
\usepackage{amsmath,amsthm,amssymb}
\usepackage[mathscr]{eucal}
 \usepackage{cite}
\usepackage{upgreek}
\usepackage[bookmarks=false]{hyperref}
\usepackage{enumerate}
\usepackage{amsmath}
\usepackage{mathrsfs}
\usepackage{blindtext}
\usepackage{scrextend}
\usepackage{enumitem}
\usepackage{bm} 
\addtokomafont{labelinglabel}{\sffamily}
\usepackage{color}
\usepackage[at]{easylist}

\usepackage{comment}

\numberwithin{equation}{section}

\theoremstyle{plain}
\newtheorem{main}{Theorem}
\newtheorem{mcor}[main]{Corollary}

\newtheorem{theorem}{Theorem}[section]

\newtheorem{lemma}[theorem]{Lemma}
\newtheorem{proposition}[theorem]{Proposition}

\theoremstyle{definition}

\newtheorem*{definition*}{Definition}

\newtheorem{remark}[theorem]{Remark}
\newtheorem{fact}[theorem]{Fact}
\newtheorem{conjecture}[theorem]{Conjecture}

\begin{document}

\title[Remarks on the diagonal embedding   and strong 1-boundedness]
{ Remarks on  the diagonal embedding   and strong 1-boundedness}

\author[Srivatsav Kunnawalkam Elayavalli]{Srivatsav Kunnawalkam Elayavalli}
\address{Institute of Pure and Applied Mathematics, UCLA, 460 Portola Plaza, Los Angeles, CA 90095, USA}\email{srivatsav.kunnawalkam.elayavalli@vanderbilt.edu}
\urladdr{https://sites.google.com/view/srivatsavke/home}

{\thanks{S.K.E was supported by a Simons Postdoctoral Fellowship. }}

\begin{abstract}
    We identify a large class of hyperbolic groups whose von Neumann algebras  are not strongly 1-bounded: Sela's hyperbolic towers over $F_2$ subgroups. We also show that any intermediate subalgebra of the diagonal embedding of $L(F_2)$ into its ultrapower doesn't have Property (T).      
\end{abstract}
\maketitle

\section{Introduction}

Voiculescu initiated a revolutionary theory of free entropy in his paper \cite{VoiculescuFreeEntropy2}. 
 The free entropy computes the asymptotic volume of the microstate spaces (matrix models approximating the distribution of a fixed  tuple in a tracial von Neumann algebra).  Voiculescu's asymptotic freeness theorem allowed him to show that the free entropy of a tuple of freely independent semicirculars is non vanishing. On the other hand, one is able to compute the free entropy when there are algebraic constraints present in the ambient algebra, such as sufficiently many commutation relations or the existence of diffuse regular subalgebras that are hyperfinite. Combining these two ideas,   Voiculescu (in \cite{Voiculescu1996}) showed that $L(F_2)$ admits no Cartan subalgebras, then   Ge (in \cite{GePrime}) obtained  using the same idea, that $L(F_2)$ is prime.  These settled problems left open by Popa in \cite{popaorthoog} where he showed that $L(F_X)$ is prime and admits no Cartan subalgebras where $X$ is an uncountable set.

One of the main modern threads of Voiculescu's free entropy theory is that of \emph{strong 1-boundedness} for von Neumann algebras, which originated with remarkable ideas of Jung in \cite{Jung2007}. Inspired by ideas from geometric measure theory, in particular Besicovitch's classification of metric spaces with Hausdorff measure 1, Jung developed technical tools to study the case when Voiculescu's free entropy dimension (a Minkowski dimension type quantity for the microstate spaces) for a tuple is 1, and discovered natural conditions wherein this property passes to the von Neumann algebra generated by the tuple. In particular, if one locates such a tuple in a von Neumann algebra, one can automatically conclude non isomorphism with $L(F_2)$, as the free entropy dimension of the semicircular generating set is 2. Jung used this in \cite{Jung2007} to prove that $L(F_2)$ cannot be generated by two  amenable (more generally strongly 1-bounded) subalgebras with diffuse intersection.

More recently, by carefully   analyzing ideas of Jung, Hayes in \cite{Hayes2018} extracted a numerical invariant of the von Neumann algebra, called the \emph{1-bounded entropy}. This framework has proved very robust and has been used to obtain several new  rigidity results for non strongly 1-bounded von Neumann algebras (such as $L(F_2)$). For instance, Hayes showed in \cite{Hayes2018} that $L(F_2)$ does not admit even quasi regular diffuse strongly 1-bounded subalgebras, generalizing the Theorem of Voiculescu.  For more recent results see \cite{FreePinsker, BenPT, PTcor, PTkilled, fullee}. For these reasons it is of great interest to identify examples of  non strongly 1-bounded von Neumann algebras.  %The following are some examples: Hayes showed in \cite{Hayes2018} that $L(F_2)$ does not admit even quasi regular diffuse strongly 1-bounded subalgebras, generalizing the Theorem of Voiculescu. Recently, in \cite{fullee}, 1-bounded entropy was used to construct II$_1$ factors without property Gamma any of whose ultrapowers are not isomorphic to any ultrapower of $L(F_2)$.  

 In this note we observe  using 1-bounded entropy, some structural   properties  of intermediate subalgebras of the diagonal embedding of $L(F_2)$ into its ultrapower.   By way of leveraging \emph{existentially closed} (see Section \ref{existential theories}) copies of $F_2$ in the group level, our first main result identifies a family of hyperbolic groups introduced by Sela (\cite{sela1}) whose von Neumann algebras are not strongly 1-bounded von Neumann algebras. Some familiar examples are hyperbolic surface groups.

\begin{main}\label{maine}
    The group von Neumann algebras of all hyperbolic towers over $F_2$ subgroups  are not strongly 1-bounded. 
\end{main}

\begin{remark}

    We thank D. Shlyakhtenko  for pointing out to us that surface group von Neumann algebras are not strongly 1-bounded is already known through a computation of the free entropy dimension, which is an apriori stronger result (see paragraph below \cite{BDJ2008} Theorem 4.13). Roughly speaking, one sees that the hyperbolic surface groups are  decomposed as an iterated amalgamated free product over copies of $\mathbb{Z}$. Then, using the free entropy dimension estimate for amalgamated free products over hyperfinite subalgebras, which is the main technical result of \cite{BDJ2008} (Theorem 4.4), one can identify a generating set using an iterative process, whose microstates free entropy dimension has a precise lower bound (in the case of genus $g$, the lower bound is $2g-1$ which is significantly greater than 1). We would like to  point out that our proof not only applies to more groups but is conceptually different and softer. Indeed, on the von Neumann algebra level we only  use the fact that $L(F_2)$  is not strongly 1-bounded, and the 1-bounded entropy inequality (see Fact \ref{in the presence of the ultra}) whose proof as outlined in \cite{Hayes2018} Proposition 4.5 is quite elementary.      
\end{remark}

Potentially, there is a larger class of groups that admit an existential copy of $F_2$, however in light of Sela's classification of groups that are elementarily equivalent to free groups, finding more examples could be hard. Note also that there also is the famous class of existentially closed groups (see \cite{ECgroups}). We document a proof that the group von Neumann algebras of these groups are   on the other hand McDuff II$_1$ factors (see Proposition \ref{}).  

By virtue of being hyperbolic, many rigidity results are already known in the setting of Theorem \ref{maine} through Ozawa's biexactness techniques  and Popa's deformation rigidity (see \cite{OzawaSolidActa, OzPopaCartan, PopaVaesHyp, ChifanSinclair}). However, one obtains using non-strong 1-boundedness,    the following stronger rigidity results below: 

\begin{mcor}
    Let $N$ be the group von Neumann algebra of a hyperbolic tower over an $F_2$ subgroup. Then the following hold: 
    \begin{enumerate}
        \item $N$ cannot be written as the join of two strongly 1-bounded subalgebras (see a comprehensive list in \ref{s1b examples}) with diffuse intersection.
        \item $N$ contains no diffuse quasi regular strongly 1-bounded subalgebra.    
    \end{enumerate}
\end{mcor}

Our next  observation describes a  surprising structural property of  intermediate subalgebras of the diagonal embedding:

\begin{main}\label{T theorem}

For any ultrafilter $\mathcal{U}$ on any set $I$,  one cannot embed a property (T) von Neumann algebra $M$ into $L(F_2)^\mathcal{U}$  such the $M$ contains the diagonal copy of $L(F_2)$. 
    
\end{main}

We would like to point out that Theorem \ref{T theorem}    works if we just replace $L(F_2)$ by any finite von Neumann algebra $N$ with $h(N\cap M: N)=\infty$ (see the first paragraph of Section \ref{1bddentropy}). Also note that any $M$ here (regardless of whether $M$ has property (T) or not) will also not have any diffuse quasi regular amenable subalgebra.  We record below a question asked to us by J. Peterson, as a conjecture:    

\begin{conjecture}
Let $N$ be a non Gamma II$_1$ factor with  the Haagerup property. Then for any ultrafilter $\mathcal{U}$ on any set $I$,  one cannot embed a property (T) von Neumann algebra $M$ into $N^\mathcal{U}$  containing the diagonal embedding.
\end{conjecture}

Note that J. Peterson's conjecture on ultrapowers (see Problem U.5 in \cite{JessePL}) implies the above conjecture. In particular Theorem \ref{T theorem}, can be seen as some evidence for the more general J. Peterson's conjecture.

We remark that in the realm of model theory of II$_1$ factors (see \cite{FHS2013, FHS2014a, FHS2014b}), $N$ being an intermediate subalgebra of $L(F_2)$ and its ultrapower  is a very well known situation.  This means that $N$ admits an \emph{existential embedding} of $L(F_2)$. Thus one of the features of this article is to emphasize  the link between model theory and strong 1-boundedness (see  also \cite{jekel}).

%We also document in this note a proof that the von Neumann algebras associated to existentially closed groups (see \cite{ECgroups})

\subsection*{Acknowledgements} The results in this article emerged from a discussion with J. Peterson. It is our pleasure to thank him and Vanderbilt University for their warm hospitality. We also sincerely thank S. André for his great assistance with some references, and guiding us through the theory of Sela's hyperbolic towers. We thank I. Chifan and T Sinclair for reading our draft carefully.

\section{Theories of groups}\label{existential theories}

We work in the language of groups. Here the  terms are words in the variables, their inverses, and the identity element (denoted by $1$). An atomic formula is two terms separated by ``$=$''.    Combining atomic formulae, logical connectors ($\lor$, $\land$, $\neg$) and quantifiers ($\forall$, $\exists$)  one describes first order formulae. When there are no free variables we call it a sentence   $\varphi$: for example, $\forall x_1\forall x_2 \exists x_3 (x_1^2x_2^3x_3=1)$.   The \emph{elementary theory} of a group  is the collection of sentences that hold true in the group.     

An existential sentence is a formula without free variables that can be expressed in the following way: $\sigma=\exists x_1 \cdots\exists x_n \ \varphi(x_1 \cdots x_n)$. The \emph{existential theory} of a group is the collection of existential sentences that hold true in the group. Note that one can also define a universal sentence, as a formula without free variables of the form $\forall   x_1 \cdots\forall x_n \ \varphi(x_1 \cdots x_n)$, and the \emph{universal theory} of a group as the collection of universal sentences that hold true in the group. 

\begin{remark}Observe that if $G$ and $H$ have the same universal theory, then they have the same existential theory because the negation of every universal sentence is an existential sentence.   

\end{remark}

An embedding of groups $\pi: H\to G$ is    said to be \emph{elementary} if for every first-order formula $\phi$ with $k$ free variables $x_1,\hdots x_k$ in the language of groups, and for every $k$-tuple $h_1,\hdots h_k\in H$, the statement $\phi(h_1,\hdots h_k)$ is true in $H$     if and only if the statement $\phi(\pi(h_1,\hdots h_k))$ is true in $G$. Similarly one defines an embedding to be \emph{existential} by replacing an arbitrary first-order formula $\phi$ with an existential first-order formula.

\section{Ultrapowers of groups and of von Neumann algebras} 

    Let $G$ be a countable group and let $1_G$ denote the identity element of $G$. For an ultrafilter $\mathcal U$ on a set $I$, we denote by $G^{\mathcal U}$ the algebraic ultraproduct: $\left(\prod_{I}G\right) /N$ where $N=\{(g_i): \{i\in I: g_i=1_G\}\in \mathcal U\}$. Note that $N$ is a normal subgroup of $\prod_{I}G$, hence $G^{\mathcal U}$ is a group. We have a natural diagonal inclusion $d_{\mathcal{U}}: G\to  G^{\mathcal U}$ given by $g\mapsto (g_i)_{\mathcal U}$, where $g_i=g$, for all $i\in I$.

The following are standard results in model theory, following from work of Kaisler-Shelah \cite{Keisler, Shelah}. For the sake of convenience assume the continuum hypothesis in this paper. 

\begin{fact}
\begin{enumerate}
\item \label{}  $G$ and $H$ are elementarily equivalent, if and only if exists an ultrafilter $\mathcal U$ on a set $I$  such that  $G^\mathcal{U}\cong H^\mathcal{U}$.  
\item \label{}  $\pi: H\to G$ is an elementary  embedding if and only if there exists an ultrafilter $\mathcal U$ on a set $I$   such that $\pi^{\mathcal{U}}: H^\mathcal{U}\to G^\mathcal{U}$ is an isomorphism.
\item \label{} $\pi: H\to G$ is an existential  embedding if and only if there exists an ultrafilter $\mathcal U$ on a set $I$  and an embedding $\rho: G\to H^{\mathcal{U}}$ such that $ \rho\circ  \pi= d_{\mathcal{U}}$.  
%and $H$ is existentially closed in $G$ then there exists an ultrafilter $\mathcal U$ on a set $I$  and an embedding of 
\end{enumerate}

\end{fact}
In this paper we will  be concerned with the objects $(M,\tau)$, tracial von Neumann algebras, i.e., a pair consisting of a von Neumann algebra $M$ and a faithful normal tracial state $\tau:M\rightarrow\mathbb C$. For any   group $G$, one denotes the group von Neumann algebra (see 1.3 of \cite{AP}) as $L(G)= \{u_g\}_{g\in G}''\subset \mathbb B(\ell^2 G)$. For any subgroup $H<G$, one has the natural inclusion of von Neumann algebras $L(H)\subset L(G)$.

For an ultrafilter $\mathcal U$ on a set $I$, we denote by $M^{\mathcal U}$ the tracial ultraproduct: the quotient $\ell^\infty(I,M)/\mathcal{J}$ by the closed ideal $\mathcal{J}\subset\ell^\infty(I,M)$  consisting of $x=(x_n)_{\mathcal U}$ with $\lim\limits_{n\rightarrow\mathcal U}\|x_n\|_2= 0$. We have the canonical trace on $M^{\mathcal U}$ given by $\tau^{\mathcal U}((x_n)_{\mathcal U})= \lim_{n\to \mathcal U} \tau(x_n)$.     
We have a natural diagonal inclusion $D_{\mathcal U}: M\to  M^{\mathcal U}$ given by $x\mapsto (x_n)_{\mathcal U}$, where $x_n=x$, for all $n\in I$.

\begin{lemma}\label{lemma main}
Let $G$ be a countable group, and $\mathcal U$ be an ultrafilter on a set $I$. Then there exists a unital $*$-homomorphism $\Theta: L(G^\mathcal{U})\to L(G)^{\mathcal{U}}$ such that $\Theta (L(d_{\mathcal{U}}(G)))= D_{\mathcal{U}}(L(G))$.   
\end{lemma}

\begin{proof}
    Let $\theta: G^\mathcal{U}\to L(G)^{\mathcal{U}}$ be    the injective homomorphism given by $\theta((g_n)_{\mathcal{U}})=(u_{g_n})_{\mathcal{U}}$. Observe that $\tau^{\mathcal{U}}((u_{g_n})_{\mathcal{U}})= \lim_{n\to \mathcal{U}} \tau(u_{g_n})=0$. Hence it follows that $\theta$  extends to a unital   $*$-homomorphism $\Theta: L(G^\mathcal{U})\to L(G)^{\mathcal{U}}$ such that $\Theta (L(d_{\mathcal{U}}(G)))= D_{\mathcal{U}}(L(G))$.    
\end{proof}

\section{Existentially closed groups and their von Neumann algebras}

A group $G$ is said to be \emph{existentially closed} if every embedding of $G$ into a group $H$ is existential. Such groups exist, and have been studied extensively in group theory. See \cite{ECgroups} for a survey. Recall also that a II$_1$ factor $M$ is McDuff (see \cite{McD1, McD2}) if  the central sequence algebra, namely $M'\cap M^\mathcal{U}$ is non abelian. The McDuff property has been very useful in the study of model theory of II$_1$ factors recently, particularly in identifying elementary equivalence classes (see \cite{FarahHartSherman-MTOA3, FGL, GoldbringHart-McDuffTheories, BCI15, AGKE, superMcD}).

\begin{proposition}
    Let $G$ be an existentially closed group. Then, $L(G)$ is a McDuff II$_1$ factor.    
\end{proposition}

\begin{proof}
    Firstly, observe that every conjugacy class of $G$ is infinite. Indeed, by contradiction suppose there is a finite conjugacy class $\{g_1,\hdots g_n\}$. Then there exists an element $h$ of $G$ such that $h\neq g_i$ for any $i\in \{1,\hdots, n\}$, and $|h|= |g_1|$, i.e, they have the same order. Indeed, let $k=|g_1|$, and consider $G\rightarrow G*\mathbb{Z}/k\mathbb{Z}$ and apply that this is an existential embedding. Then, one sees that $h$ and $g_1$ are conjugate by considering the HNN extension and applying again the existential property.     Secondly, consider the embedding of $\pi: G \rightarrow G\times (G*G)$ given by $g\mapsto (g,1_{G*G})$. Since $G$ is existentially closed, $\pi$ is existential. Hence for any finite set $F\subset G$, there exists $g_{F,1},g_{F,2}\in G$ such that $[F,g_{F,i}]=1_{G}$ for all $i\in \{1,2\}$ and $[g_{F,1},g_{F,2}]\neq 1_G$. This implies that $L(G)$ is McDuff. Indeed fix $\mathcal{U}$ a non principal ultrafilter on $\mathbb{N}$. Let $F_n$ be an increasing family of finite subsets of $G$ such that $\bigcup_{n}F_n=G$.  Then see that $(g_{F_n,i})_{\mathcal{U}}\in L(G)'\cap  L(G)^{\mathcal{U}}$ for all $i\in \{1,2\}$ and also $[(g_{F_n,1})_{\mathcal{U}},     (g_{F_n,2})_{\mathcal{U}}]\neq 0$. Hence we are done.          
\end{proof}

A separable tracial von Neumann algebra $N$ is called \emph{existentially closed} (see \cite{ECfactors}) if for  unital inclusion $N\subset M$ into a von Neumann separable tracial von Neumann algebra $M$, there exists an ultrafilter $\mathcal{U}$ on a set $I$ and an embedding $\pi: M\to N^\mathcal{U}$ such that $\pi$ restricted to  $N$ coincides with $D_{\mathcal{U}}(N)$.  It is well known that existentially closed separable tracial von Neumann algebra exist and are McDuff (see \cite{ECMcDuff}). In the context of the above Proposition, it is natural to wonder if $L(G)$ is an existentially closed tracial von Neumann algebra, when $G$ is an existentially closed group.

\section{1-bounded entropy}\label{1bddentropy}
For a finite tuple $X$ of self-adjoint operators in a tracial von Neumann algebra $(M,\tau)$, one has the {\it $1$-bounded entropy } $h(X)$, implicit in Jung's work \cite{Jung2007} and defined explicitly by Hayes \cite{Hayes2018}. It is the exponential growth rate of the covering numbers of Voiculescu's microstate spaces (see \cite{VoiculescuFreeEntropy2}) up to unitary conjugation.  For an inclusion $N\subset M$ of tracial von Neumann algebras, the \emph{$1$-bounded entropy of $N$ in the presence of $M$}, denoted $h(N:M)$, is defined by modifying the definition of $1$-bounded entropy to only measure the size of the space of microstates for $N$ which have an extension to microstates for $M$. See Section 2.2 and 2.3 in \cite{HJKE1} for a detailed and rigorous exposition. Note that $h(X_1:M)=h(X_2:M)$ if $X_1''=X_2''$ by \cite[Theorem A.9]{Hayes2018}. Hence, given a von Neumann subalgebra $N\subset M$, we unambiguously write $h(N:M)$ (and call it the {\it $1$-bounded entropy of $N$ in the presence of $M$}) to be $h(X:Y)$ for some generating sets $X$ of $N $ and $Y$ of $M$. 
We write $h(M)=h(M:M)$ and call it the {\it $1$-bounded entropy of $M$}.
 %In words, $h$ is the exponential growth rate of the covering numbers of Voiculescu's microstate spaces (see \cite{VoiculescuFreeEntropy2}) up to unitary conjugation.  For an inclusion $N\leq M$ of tracial von Neumann algebras, the \emph{$1$-bounded entropy of $N$ in the presence of $M$}, denoted $h(N:M)$, is defined by modifying the definition of $1$-bounded entropy to only measure the size of the space of microstates for $N$ which have an extension to microstates for $M$. See Section 2.2 and 2.3 in \cite{HJKE1} for a detailed and rigorous exposition.
 
 For the purposes of this article we recall the following facts about $h$: 
 
 \begin{fact}(see  \cite[2.3.3]{HJKE1})\label{fact 1}
 $h(N_{1}:M_{1})\leq h(N_{2}:M_{2})$ if $N_{1}\subset N_{2}\subset M_{2}\subset M_{1}$ and $N_{1}$ is diffuse.
 \end{fact}

   \begin{fact}(see \cite[Proposition 4.5]{Hayes2018})\label{in the presence of the ultra}
 $h(N:M)=h(N:M^{\mathcal U})$ if $N\subset M$ is diffuse, and $\mathcal U$ is an ultrafilter on a set $I$. (Note that \cite[Proposition 4.5]{Hayes2018} asserts this fact for free ultrafilters $\mathcal U$. The fact is trivially true also for non-free (i.e., principal) ultrafilters.) \end{fact}

\begin{fact}(see \cite[Lemma 3.7]{Jung2007})) $h(N_1*N_2)=\infty $ where $(N_1,\tau_1)$ and $(N_2,\tau_2)$ are Connes-embeddable diffuse tracial von Neumann algebras. In particular $h(L(F_2))=\infty$. \end{fact}

\begin{fact}(see \cite[Proposition A.16] {Hayes2018}))\label{in  presence of the ultra} $N$ is strongly 1-bounded in the sense of Jung (\cite{Jung2007}) if and only if $h(N)<\infty$.\end{fact}

The following are examples of strongly 1-bounded von Neumann algebras:

\begin{fact} \label{s1b examples}
    \begin{enumerate}
        \item Diffuse amenable tracial von Neumann algebras (see \cite{Jung2007}).
        \item Tensor products of two diffuse tracial von Neumann algebras (see \cite{Jung2007}). 
        \item Tracial von Neumann algebras that have a diffuse hyperfinite quasi regular subalgebra (see Theorem 3.8 in \cite{Hayes2018}). 
        \item \label{Kzahdan}Tracial von Neumann algebras that have a Kazhdan set (see \cite{HJKE1}). This includes all property (T) II$_1$ factors, and property (T) group von Neumann algebras.
        \item Group von Neumann algebras of groups that have vanishing first $L^2$ Betti number and are sofic and finitely presented (see \cite{JungL2B, Shl2015, HJKE2}). 
        \item Free orthogonal quantum group von Neumann algebras (see \cite{BrannanVergnioux}). 
        \item The non Gamma factors  constructed in \cite{fullee} and their ultrapowers.
    \end{enumerate}
    
\end{fact}

\section{Existential embeddings of $F_2$ }

%The study of the  elementary theory of non abelian free groups has evolved into a pillar of modern of group theory, thanks to deep works of Sela (see his series of papers starting from \cite{sela1}), and several other mathematicians.  The point of this subsection is to document examples of groups that admit existential inclusions of the free group on two generators, $F_2$. 

A sequence of homomorphisms $\{\pi_n:G \to H\}$, $n\in \mathbb{N}$, is \emph{eventually faithful} if for every $g\neq 1_G$, there exists an $N\in \mathbb{N}$ such that for all $n>N$, $\pi_n(g)\neq 1_H$.

\begin{lemma}[see also Lemma 3.1 of \cite{andre}]\label{andre}

Let $G$ be a finitely generated group and let $H$ be a finitely generated subgroup of $G$. If there is a sequence of eventually faithful homomorphisms $\{\pi_n:G \to H\}$, such that ${\pi_n}(h)=h$ for all $h\in H, n\in \mathbb{N}$, then there exists a non principal ultrafilter $\mathcal{U}$ on $\mathbb{N}$ and an embedding $\pi: G\to H^{\mathcal{U}}$ such that $\pi_{|H}= d_{\mathcal{U}}(H)$. 
    
\end{lemma} 

\begin{proof}
      Fix a a non principal ultrafilter $\mathcal{U}$ on $\mathbb{N}$. Define the map $\pi: G\to H^{\mathcal{U}}$ as follows: $\pi(g)= (\pi_n(g))$. Clearly, $\pi(h)= (h)_{\mathcal U}$ and hence, $\pi_{|H}= d_{\mathcal{U}}(H)$. Moreover, the fact that $\pi$ is a homomorphism is checked easily since   $\pi_n$ is a homomorphism for all $n\in \mathbb{N}$. Also, $\pi$ is injective because for all $g\in G$, $\{\pi_n(g)\neq 1_H\}\in \mathcal{U}$.    
\end{proof}

One can construct by hand such an eventually faithful sequence of homomorphisms in the case that $G$ is a hyperbolic surface group and $H$ is a particular non abelian free subgroup. 

\begin{lemma}[see Corollary 2.2 of \cite{Breuillard_2006}]\label{ultra}

Let $G$ be the fundamental group of a surface of an orientable surface of genus $2r$. Consider a presentation of $G$ as follows: 
$$G= \langle    a_i,a_i',b_i,b_i', 1\leq i\leq r\ |\ \prod_{i=1}^r [a_i,a_i']\prod_{j=1}^r [b_j,b_j']=1_G\rangle.$$

Then consider the automorphism $\sigma$ of $G$ that leaves $a_i$ and $a_i'$ fixed for all $1\leq i\leq r$ and sends every $b_i$ to $gb_ig^{-1}$ and every $b_i'$ to $gb_i'g^{-1}$ where $g=\prod_{i=1}^r [a_i,a_i']$. Abuse notation and let  $\pi: G\to \langle a_1,\hdots a_r,a_1',\hdots a_r' \rangle $ given by $\pi(a_i)= \pi(b_i)= a_i$ and $\pi(a_i')= \pi(b_i')= a_i'$. Then the sequence of homomorphisms $\{\pi\circ \sigma^n\}_{n\in \mathbb N}$ is eventually faithful. 
\end{lemma}

More generally than this discussion one can describe a class of groups that admit an elementary inclusion of $F_2$, which in particular satisfies the conclusion of Lemma \ref{andre}. For the definition of this family, we direct the reader to \cite{Perin, Perinthesis}.  See Section A.1 and the Proof of Corollary A.2 (of \cite{guirardel2020towers}), where an eventually faithful sequence of homomorphisms in the case that $G$ is a hyperbolic tower over an $F_2$ subgroup is constructed:     

\begin{theorem}\label{sela hyperbolic tower}[\cite{guirardel2020towers}]
    Let $G$ be a hyperbolic tower (in the sense of Sela \cite{sela1}) over an $F_2$ subgroup.  Then this inclusion of $F_2$ in $G$  satisfies the conclusion of Lemma \ref{andre}.
\end{theorem}

%One sees this by taking $H=F_2 $ in Figure 1 of \cite{Perin}, because if $G$ is a hyperbolic tower over $H$ then there is a discriminating sequence of retractions $\{\phi_n : G \rightarrow H\}_n$. The proof of this fact is more or less the same as the proof of Corollary 2.2 in \cite{Breuillard_2006}. The proof then follows from using  Lemma 3.1.    

\section{Proofs}

\begin{proof}[Proof of Theorem \ref{maine}]

Let $G$ be a hyperbolic  tower over $H= F_2$. Then from Theorem \ref{sela hyperbolic tower} and  Lemma \ref{andre}, we see that there exists  a non principal ultrafilter $\mathcal{U}$ on $\mathbb{N}$, $H<G$ and an embedding $\pi: G\to H^{\mathcal{U}}$ such that $\pi_{|H}= d_{\mathcal{U}}(H)$. Now applying Lemma \ref{lemma main}, we see that there exists a unital $*$-homomorphism $\Theta: L(G)\to L(H)^{\mathcal{U}}$ such that $\Theta (L(d_{\mathcal{U}}(H)))= D_{\mathcal{U}}(L(H))$.  Now applying Facts \ref{fact 1}, \ref{in the presence of the ultra} we see the following inequality: $\infty= h(L(F_2))= h(L(H))= h(L(H): L(H)^{\mathcal{U}})\leq h(L(G): L(H)^{\mathcal{U}})\leq h(L(G))$. Finally, applying Fact  \ref{in  presence of the ultra}  we are done.

\end{proof}

We record the following proposition which is contained in the above proof:

\begin{proposition}
    If there exists $F_2<G$ an existential embedding for some group $G$, then $L(G)$ is not strongly 1-bounded.    More generally, if $N$ is not strongly 1-bounded, and $N\subset M$ is an existential embedding, then $M$ is not strongly 1-bounded.      

\end{proposition}

\begin{proof}[Proof of Theorem \ref{T theorem}]

From the same inequality as above, we have  $\infty= h(L(F_2))= h(L(F_2): L(F_2)^{\mathcal{U}})\leq h(M: L(F_2)^{\mathcal{U}})\leq h(M)$. However, since $L(F_2)$ does not have property Gamma, it follows that $M$  is a factor, and hence admits a Kazhdan set (see Proposition 1 in \cite{ConnesJones}). By Fact \ref{s1b examples}(\ref{Kzahdan}), we get a contradiction.     
    
\end{proof}

\begin{remark}
    From  Theorem \ref{T theorem},  one actually obtains for every inclusion  of $L(F_2)\subset M$, where $M$ has property (T), a model theoretic sentence $\phi$ with $n$ free variables such that $\phi_M(x_1\hdots x_n)=0$ and $\phi_{L(F_2)} (x_1\hdots x_n)>0$ for some tuple $x_i\in L(F_2)$, $i=1,\hdots n$.  
\end{remark}

\bibliographystyle{amsalpha}
\bibliography{inneramen}

\end{document}